\documentclass[10pt]{amsart}
\usepackage[margin=1in]{geometry}
\usepackage{amsmath}
\usepackage{amssymb}
\usepackage{amsthm}
\usepackage{amscd}
\usepackage{color}
\usepackage{hyperref}
\usepackage{url}
\usepackage{mathtools}
\usepackage{graphicx}
\usepackage[all, knot]{xy}

\theoremstyle{plain}
\newtheorem{theorem}{Theorem}
\newtheorem{corollary}[theorem]{Corollary}
\newtheorem{proposition}[theorem]{Proposition}
\newtheorem{lemma}[theorem]{Lemma}

\theoremstyle{definition}
\newtheorem{definition}[theorem]{Definition}

\newcommand{\TC}{\mathbf{TC}}
\newcommand{\LTC}{\mathbf{LTC}}
\newcommand{\tc}{\mathbf{tc}}
\newcommand{\ltc}{\mathbf{ltc}}
\newcommand{\Sch}{\mathbf{Sch}}
\newcommand{\cat}{\mathbf{cat}}
\newcommand{\nil}{\mathbf{nil}}

\begin{document}
\begin{center}
\Large Lusternik-Schnirelmann category and based topological complexities of motion planning
\end{center}

\vspace{0.3in}

\centerline{Yongheng Zhang}

\vspace{0.3in}

\centerline{ABSTRACT}

Farber and Rudyak introduced topological complexity $\TC(X)$ of motion planning and its higher analogs $\TC_n(X)$ to measure the complexity of assigning paths to point tuples. Motivated by motion planning where a robotic system starts at the home configuration and possibly comes back after passing through a list of locations, we define three other classes of topological complexities $\LTC_n(X)$, $\ltc_n(X)$ and $\tc_n(X)$.  We will compare these notions and compute the latter for some familiar classes of spaces.\\


\section{Introduction}

Given $n\geq 2$ ordered points in a path-connected space $X$, we can construct a path starting at the first point, successively passing through the others points in order before ending at the last point. We also want such paths to vary continuously with respect to different $n-$tuples of points in $X$. When $X$ is contractible, this can be done globally. But when $X$ is not topologically trivial, one has to find an open covering of $X$ with cardinality greater than one such that a continuous assignment of paths to point tuples is possible over each open set in the covering. Topological complexity $\TC(X)$ \cite{Far1} and its higher analogs $\TC_n(X)$ ($\TC_2(X)=\TC(X)$) \cite{Rud} were introduced to measure the least such cardinality.\\

Robotic motion planning motivated the introduction of topological complexity \cite{Far1}, where $X$ is the configuration space of a robotic system. In some practical situations, a system starts at a home position, then moves to the locations in a prescribed list for task completion. Sometimes, the system also comes back to the home position. This paper introduces topological complexities $\tc_n(X)$ for motion planning which starts at a base point and also the versions $\LTC_n(X)$\footnote{$\LTC_2(X)$ was introduced earlier by My Ismail Mamouni and Derfoufi Younes.} and $\ltc_n(X)$ for systems traversing a loop. We will compare these four versions of topological complexities and show that $\ltc_n(X)$ and $\tc_n(X)$ are directly related to the Lusternik-Schnirelmann category. We will also compute them for some familiar spaces.\\

In this paper, $X$ is a path-connected topological space; $x_0$ is a chosen point in $X$, which denotes the home configuration when $X$ is the configuration space of a robotic system; $I$ is the unit closed interval. We denote the $n$th Cartesian power of $X$ by $X^n$. All maps that will be considered are continuous.\\

\section{The four versions of topological complexities}
Recall the definitions of the following four spaces, which are equipped with the compact open topologies.\\

\begin{equation*}
\begin{array}{llll}
\mbox{Based path space:} & P_{x_0}X &=&\{f: I\rightarrow X\big| f(0)=x_0\}.\\
\mbox{Based loop space:} & \Omega_{x_0}X &=& \{f:I\rightarrow X\big| f(0)=f(1)=x_0\}.\\
\mbox{Free path space:} & PX &=& \{f:I\rightarrow X\big\}.\\
\mbox{Free loop space:} & LX &=& \{f:I\rightarrow X\big| f(0)=f(1)\}.\\
\end{array}
\end{equation*}\\

Let $n$ be a positive integer, then we define the following maps.\\

\begin{equation*}
\begin{array}{crcll}
p_n:&P_{x_0}X&\longrightarrow &X^{ n},&p_n(f)=\left(f(\frac{1}{n}),f(\frac{2}{n}),\cdots,f(\frac{n-1}{n}),f(1)\right).\\
q_n:&\Omega_{x_0}X&\longrightarrow & X^{ n}, & q_n(f)=\left(f(\frac{1}{n+1}),f(\frac{2}{n+1}),\cdots,f(\frac{n}{n+1})\right).\\
\end{array}
\end{equation*}\\

Let $n$ be an integer such that $n\geq 2$. We define two more maps.\\

\begin{equation*}
\begin{array}{crcll}
P_n:&PX&\longrightarrow & X^{ n}, &P_n(f)=\left(f(0),f(\frac{1}{n-1}),f(\frac{2}{n-1}),\cdots,f(\frac{n-2}{n-1}),f(1)\right).\\
Q_n:&LX&\longrightarrow & X^{ n}, &Q_n(f)=\left(f(0),f(\frac{1}{n}),f(\frac{2}{n}),\cdots,f(\frac{n-1}{n})\right).\\
\end{array}
\end{equation*}\\

Since $X$ is path connected, all maps above are surjective. In fact, they are fibrations in the sense introduced in \cite{Ser}.

\begin{proposition}
The maps $p_n$, $q_n$, $P_n$ and $Q_n$ are (Hurewicz) fibrations.
\end{proposition}

\begin{proof}[Proof.] Let $Y$ be any topological space. We want to show that these maps satisfy the covering homotopy property with respect to $Y$. We will see that this follows from the existence of map extensions from $Y\times C$, where $C$ is a comb-shaped space embedded in $I^{\times 2}$, to $Y\times I^{\times 2}$. Since the four cases are similar, we will only prove it for $p_n$ and leave the proofs of the other three to the reader for routine check.\\

So let $g: Y\rightarrow P_{x_0}X$ and $h: Y\times I\rightarrow X^{ n}$ be maps such that $(p_n\circ g)(y)=h(y,0)$. These are equivalent to maps $G:Y\times I\rightarrow X$ and $h_1,h_2,\cdots,h_n:Y\times I\rightarrow X$ such that 
\begin{equation}
\label{*}
\tag{$*$}
G(y,0)=x_0\mbox{ and }G(y,\frac{i}{n})=h_i(y,0), i=1,2,\cdots,n.
\end{equation}

We want to show that there is a map $\widetilde{h}:Y\times I\rightarrow P_{x_0}X$ such that $\widetilde{h}(y,0)=g(y)$ and $p_n\circ \widetilde{h}=h$. This is equivalent to a map $\widetilde{H}:Y\times I\times I\rightarrow X$ such that  $\widetilde{H}(y,t,0)=x_0$, $\widetilde{H}(y,0,s)=G(y,s)$ and $\widetilde{H}(y,t,\frac{i}{n})=h_i(y,t)$, $i=1,2,\cdots,n$.\\

Let $C$ be the union of the subspaces $I\times \{\frac{j}{n}\}$, $j=0,1,2,\cdots,n$ and $\{0\}\times I$ of $I^{\times 2}$. The compatibility condition \eqref{*} tells us that such an $\widetilde{H}$ exists on $Y\times C$. It is not difficult to extend this $\widetilde{H}$ to be on the entire $Y\times I^{\times 2}$. An example is given as follows, where $h_0(y,t)$ is defined to be $x_0$.\\

\begin{equation*}
\widetilde{H}(y,t,s)=\left\{
\begin{array}{ll}
 h_{j+1}(y,t_*)&\mbox{ if } \frac{j+1}{n}-\frac{2t}{5n}\leq s\leq \frac{j+1}{n},\\
G(y,\frac{5s-\frac{4j+2}{n}t}{5-4t}) & \mbox{ if } \frac{j}{n}+\frac{2t}{5n}\leq s\leq \frac{j+1}{n}-\frac{2t}{5n},\\
h_j(y,t_{**})&\mbox{ if } \frac{j}{n}\leq s \leq  \frac{j}{n}+\frac{2t}{5n},
\end{array}\right. j=0,1,2,\cdots,n-1,
\end{equation*}

where		
\begin{equation*}
\begin{array}{rcl}
t_*&= & \frac{5ns-5(j+1)+2+t-\sqrt{(5ns-5(j+1)+2+t)^2-4(5ns-5(j+1)+2t)}}{2},\\
t_{**}&=& \frac{-5ns+5j+2+t-\sqrt{(5ns-5j-2-t)^2+4(5ns-5j-2t)}}{2}.
\end{array}
\end{equation*}

\end{proof}

\begin{definition}
Let $p:E\rightarrow B$ be a fibration. The genus of $p$ \cite{Sch}, or the sectional category of $p$ \cite{James}, is defined to be the smallest number of open sets such that these open sets cover $B$ and  when restricted to each open set, $p$ admits a section (we call it a local section). If no such number exists, the genus is defined to be $\infty$. We denote the genus of $p$ by $\Sch(p)$, since it was introduced by Schwarz (\v{S}varc).\\
\end{definition}

\begin{definition} The topological complexity  $\TC(X)$ was defined by Farber \cite{Far1} \cite{Far2} as follows.
$$\TC(X)=\Sch(P_2).$$
For each $n\geq 2$, higher topological complexites  $\TC_nX$ were defined by Rudyak (See Remark 3.2.5 of \cite{Rud}): $$\TC_n(X)=\Sch(P_n).$$
We define three other versions of topological complexities below.

\begin{equation*}
\begin{array}{ll}
\mbox{Loop topological complexities:} & \LTC_n(X)=\Sch(Q_n), n\geq 2.\\
\mbox{Based topological complexities:} & \tc_n(X)=\Sch(p_n), n\geq 1.\\
\mbox{Based loop topological complexities:} & \ltc_n(X)=\Sch(q_n), n\geq 1.\\
\end{array}
\end{equation*}\\

\end{definition}

Similar to $\TC_n$, the other three are also topological invariants.\\

\begin{proposition}
\label{PropositionInv}
Let $Y$ be homotopy equivalent to $X$ and $y_0$ a chosen point in $Y$. Then $\TC_n(X)=\TC_n(Y)$, $\LTC_n(X)=\LTC_n(Y)$, $\tc_n(X)=\tc_n(Y)$ and $\ltc_n(X)=\ltc_n(Y)$.
\end{proposition}

\begin{proof}[Proof.] The proof is adapted from \cite{Far1}. We will only prove it for $\tc_n$. The other cases are similar.  Since $X$ an $Y$ are homotopy equivalent, there are maps $f: X\rightarrow Y$, $g: Y\rightarrow X$, and $H:X\times I\rightarrow X$ such that $H(x,0)=x$ and  $H(x,1)=(g\circ f)(x)$. Let $V$ be an open set in $Y^{\times n}$ such that there is a local section $s:V\rightarrow P_{y_0}Y$.  Using $s$ and $H$, we will construct a local section $s'$ on the open subset $U:=(f^{\times n})^{-1}(V)$ of $X^{ n}$. Because the inverse images of the sets in an open covering of $Y^{\times n}$ cover $X^{ n}$, it then follows that $\tc_n(X)\leq \tc_n(Y)$. Similarly, $\tc_n(X)\geq \tc_n(Y)$ and thus $\tc_n(X)= \tc_n(Y)$.\\

To construct the section $s'$, first let $\phi$ be a path in $Y$ satisfying $\phi(0)=f(x_0)$ and $\phi(1)=y_0$. If $0\leq t\leq \frac{1}{n}$, then we define 

\begin{equation*}
s'(x_1,x_2,\cdots,x_n)(t)=\left\{
\begin{array}{ll}
H(x_0,4nt) &\mbox{ if } 0\leq t\leq \frac{1}{4n},\\
(g\circ\phi)(4nt-1) &\mbox{ if }  \frac{1}{4n}\leq t\leq\frac{1}{2n},\\
(g\circ s(f(x_1),f(x_2),\cdots,f(x_n)))(4t-\frac{2}{n}) &\mbox{ if } \frac{1}{2n}\leq t\leq \frac{3}{4n},\\				
H(x_1,4-4nt) &\mbox{ if } \frac{3}{4n}\leq t\leq \frac{1}{n}.
\end{array}\right.
\end{equation*}

For $\frac{i}{n}\leq t\leq \frac{i+1}{n}$ where $i=1,2,\cdots,n-1$, we define

\begin{equation*}
s'(x_1,x_2,\cdots,x_n)(t)=\left\{
\begin{array}{ll}
H(x_i,3nt-3i)& \mbox{ if } \frac{i}{n}\leq t\leq \frac{3i+1}{3n},\\
(g\circ s(f(x_1),f(x_2),\cdots,f(x_n)))(3t-\frac{2i+1}{n})&\mbox{ if } \frac{3i+1}{3n}\leq t\leq\frac{3i+2}{3n},\\
H(x_{i+1},3+3i-3nt)&\mbox{ if } \frac{3i+2}{3n}\leq t\leq \frac{i+1}{n}.
\end{array}\right.
\end{equation*}

\end{proof}

\section{Relationships among $\TC_n$, $\LTC_n$, $\tc_n$ and $\ltc_n$}

The four notions of topological complexities are not all different. Before we state their relationships, let us recall a fact from \cite{Sch}.

\begin{lemma}
\label{Lemma1}
Let $p:E\rightarrow B$ and $p':E'\rightarrow B$ be fibrations over the same base space. Let $f: E\rightarrow E'$ be a map such that $p'\circ f=p$. Then  $\Sch(p')\leq \Sch(p)$.
\end{lemma}
\begin{proof}[Proof.] Let $U$ be an open set of $B$ such that there is a local section $s:U\rightarrow E$ for $p$. Then $f\circ s:U\rightarrow E'$ is a local section for $p'$. Therefore, $\Sch(p')\leq \Sch(p)$.
\end{proof}

Using this lemma, we will see that the four topological complexities reduce to two.

\begin{proposition}
$\ltc_n(X)=\tc_n(X)$ and $\LTC_n(X)=\TC_n(X)$.
\end{proposition}

\begin{proof}[Proof.]
Let us prove $\ltc_n(X)=\tc_n(X)$. The other case is similar.\\

Let $f: P_{x_0}X\rightarrow \Omega_{x_0}X$ be defined by
\begin{equation*}
f(\phi)(t)=\left\{
\begin{array}{ll}
\phi(\frac{n+1}{n}t)& \mbox{ if } 0\leq t\leq \frac{n}{n+1},\\
\phi((n+1)(1-t)) & \mbox{ if } \frac{n}{n+1}\leq t\leq 1.
\end{array}\right.
\end{equation*}

Then $p_n=q_n\circ f$. By Lemma \ref{Lemma1}, $\ltc_n(X)\leq \tc_n(X)$.\\

On the other hand, define $g: \Omega_{x_0}X\rightarrow P_{x_0}X$ by 
$$g(\phi)(t)=\phi(\frac{n}{n+1}t).$$

Then $q_n=p_n\circ g$. By Lemma \ref{Lemma1}, $\ltc_n(X)\geq \tc_n(X)$. Therefore, $\ltc_n(X)=\tc_n(X)$.

\end{proof}

\textbf{Remark.} $\LTC_2(X)=\TC(X)$ was first obtained by My Ismail Mamouni and Derfoufi Younes.\\

So $\LTC_n$ and $\ltc_n$ do not introduce new topological invariants. Starting from now on, we will only study $\tc_n$ and its relation to $\TC_n$. First of all, we have the same result as $\TC_n(X)\leq \TC_{n+1}(X)$ \cite{Rud} for $\tc_n$. And its proof is not different from that of \cite{Rud}.\\

\begin{proposition}
\label{Proposition2}
$\tc_n(X)\leq \tc_{n+1}(X)$.
\end{proposition}

\begin{proof}[Proof.]
Let $U$ be an open set in $X^{ n+1}$ and $s:U\rightarrow P_{x_0}X$ a local section for $p_{n+1}: P_{x_0}X\rightarrow X^{ n+1}$. Let $q$ be the projection of $X^{n+1}$ onto the last factor. Then choose a point $x_*$ in $q(U)\subset X$. Define the map $\mathbf{i}:X^{ n}\rightarrow X^{ n+1}$ by $\mathbf{i}(x_1,x_2,\cdots,x_n)=(x_1,x_2,\cdots,x_n,x_*)$. We also define $f:P_{x_0}X\rightarrow P_{x_0}X$ by $f(\phi)(t)=\phi(\frac{n}{n+1}t)$. Then the map $f\circ s\circ\mathbf{i}:\mathbf{i}^{-1}(U)\rightarrow P_{x_0}X$ is a local section of $p_n$.\\

Let $\{U_i\}$ be an open covering of $X^{ n+1}$ such that there are local sections over them for $p_{n+1}: P_{x_0}X\rightarrow X^{ n+1}$, then $\{\mathbf{i}^{-1}(U_i)\}$ is an open covering of $X^{ n}$ and there are local sections over them for $p_n: P_{x_0}X\rightarrow X^{ n}$ as constructed above. It then follows that $\tc_n(X)\leq \tc_{n+1}(X)$. 
\end{proof}

\begin{proposition} 
\label{Proposition1}
When $n\geq 2$, $\TC_n(X)\leq \tc_n(X)$.
\end{proposition}

\begin{proof}[Proof.]
Define $f:P_{x_0}X\rightarrow PX$ by $$f(\phi)(t)=\phi(\frac{n-1}{n}t+\frac{1}{n}).$$ Then $ P_n\circ f=p_n$. By Lemma \ref{Lemma1}, $\TC_n(X)\leq \tc_n(X)$.
\end{proof}

\textbf{Remark.} In general, there does not exist $g:PX\rightarrow P_{x_0}X$ so as to obtain $\tc_n(X)\leq \TC_n(X)$ using Lemma \ref{Lemma1}. In fact, to get such a $g$, one has to be able to find a path connecting any point in $X$ to $x_0$ continuously. But this is precisely obstructed by $\tc_1(X)$, which is greater than one unless $X$ is contractible. Nevertheless, we do have $\tc_{n-1}(X)\leq \TC_n(X)$. With Proposition \ref{Proposition5}, this also leads to Proposition \ref{Proposition2}.\\

\begin{proposition}
\label{Proposition5}
If $n\geq 2$, then $\tc_{n-1}(X)\leq \TC_n(X)$.
\end{proposition}

\begin{proof}[Proof.]
Let  $U_i$ be open sets in $X^{n}$ such that each $U_i$ has nonempty intersection with $\{x_0\}\times X^{n-1}$ and $\{x_0\}\times X^{n-1}$ is contained in the union of these $U_i$. Let $s_i:U_i\rightarrow PX$ be local sections for $P_n:PX\rightarrow X^{ n}$. Define the map $\mathbf{j}:X^{ n-1}\rightarrow X^n$ by $\mathbf{j}(x_1,x_2,\cdots,x_{n-1})=(x_0,x_1,\cdots,x_{n-1})$. Then each composite $s_i\circ \mathbf{j}$ is actually a map from $\mathbf{j}^{-1}(U_i)$ to $P_{x_0}X$. In fact, it is a local section of $p_{n-1}:P_{x_0}X\rightarrow X^{ n-1}$. Thus, we have $\tc_{n-1}(X)\leq \TC_n(X)$.
\end{proof}

\textbf{Remark.} Combined with Corollary \ref{Corollary1}, this result was first known in \cite{BGRT}.\\

\section{$\tc_n$ and Lusternik-Schnirelmann category}

Unlike $\TC_n$, which are new topological invariants, $\tc_n$ are directly related to the classical notion of Lusternik-Schnirelmann category.\\

\begin{definition}
$\cat(X)$, the Lusternik-Schnirelmann category of $X$, is defined to be the least cardinality of an open covering of $X$ in which each open set is contractible in $X$.
\end{definition}

Let us recall a theorem from \cite{Sch}.\\

\begin{theorem}
\label{Theorem1}
Let $p:E\rightarrow B$ be a fibration. Then, $\Sch(p)\leq \cat(B)$ and the equality holds if $E$ is contractible.
\end{theorem}
\begin{proof}[Proof.] The first part uses the covering homotopy property of fibrations. Let $U_i$, $i=1,2,\cdots,\cat(B)$ be an open covering of $B$ such that there are $b_i\in B$ and  $h_i:U_i\times I\rightarrow B$ satisfying $h_i(x,0)=b_i$ and $h_i(x,1)=x$. Let $e_i\in p^{-1}(b_i)$ and $f_i: U_i\rightarrow E$ be defined by $f_i(x)=e_i$. Thus, $p\circ f_i(x)=h_i(x,0)$. Then the covering homotopy property says there are $\widetilde{h}_i:U_i\times I\rightarrow E$ such that $p\circ \widetilde{h}_i=h_i$. Thus, $p\circ \widetilde{h}_i(x,1)=h_i(x,1)=x$. So $s_i:U_i\rightarrow E$ defined by $s_i(x)=\widetilde{h}_i(x,1)$ is a section on $U_i$.   Therefore, $\Sch(p)\leq \cat(B)$.\\

Now let $E$ be contractible, i.e., there is $e_0\in E$ and $H:E\times I\rightarrow E$ such that $H(x,0)=x$ and $H(x,1)=e_0$. Let $V_i$, $i=1,2,\cdots,\Sch(p)$ be an open covering of $B$ and $s_i:V_i\rightarrow E$ local sections. Then $G_i:V_i\times I\rightarrow B$ defined by $G_i(x,t)=p\circ H(s_i(x),t)$ satisfies $G_i(x,0)=p\circ s_i(x)=x$ and $G_i(x,1)=p(e_0)$, i.e., $V_i$ is contractible in $B$. Thus, $\cat(B)\leq \Sch(p)$. Combining with the above fact, we have $\Sch(p)=\cat(B)$.
\end{proof}

Recall that the total space in $p_n:P_{x_0}X\rightarrow X^{ n}$ is contractible \cite{Ser}. Therefore, by Theorem \ref{Theorem1}, we have the following result.\\

\begin{corollary}
\label{Corollary1}
 $$\tc_n(X)=\cat(X^{ n}).$$
\end{corollary}

\textbf{Remark.}  Theorem \ref{Theorem1} applied to $P_n: PX\rightarrow X^{ n}$ and Corollary \ref{Corollary1} also tell us that $\TC_n(X)\leq \tc_n(X)$ when $n\geq 2$, which was obtained in Proposition \ref{Proposition1}.

\section{Examples of $\tc_n(X)$}

We will compute $\tc_n(X)$, i.e., $\ltc_n(X)$ for several familiar families of spaces (CW-complexes). Some of them are naturally configuration spaces of robotic systems. Since $\tc_n(X)=\cat(X^n)$, let us first recall some theorems for the Lusternik-Schnirelmann category. A good reference is \cite{CLOT}. All proofs can be found in it.\\

First of all, the dimension and higher connectedness of a CW-complex provide an upper bound for $\cat$. Using an equivalent formulation of the Lusternik-Schnirelmann category due to Whitehead, one has the following theorem.\\

\begin{theorem}
\label{TheoremW}
If $X$ is an $(n-1)-$connected ($\pi_i(X)=0$ for $i=0,1,\cdots,n-1$) CW-complex for $n\geq 1$, then $$\cat(X)\leq \mathrm{dim}(X)/n+1.$$
\end{theorem}

On the other hand, a lower bound for $\cat(X)$ is provided by the index of nilpotency of cohomology rings of $X$.\\

\begin{theorem} 
\label{TheoremN}
Given a commutative ring $R$ and consider the reduced cohomology ring $\widetilde{H}^*(X;R)$. The index of nilpotency of $\widetilde{H}^*(X;R)$, denoted by $\nil_R(X)$, is the least integer $N$ such that $\left(\widetilde{H}^*(X;R)\right)^N=0$ under the cup prodcut. Then for any $R$, $\nil_R(X)\leq\cat(X)$.\\
\end{theorem}

The following theorem relates the categories of different $X$ powers.\\

\begin{theorem}
\label{TheoremP}
If $X$ and $Y$ are CW-complexes, then $\cat(X\times Y)-1\leq (\cat(X)-1)+(\cat(Y)-1)$. Thus, $\cat(X^n)\leq n (\cat(X)-1)+1$.\\
\end{theorem}

\begin{corollary}
\label{Corollary2}
By Theorem \ref{TheoremW} and \ref{TheoremP}, if $X$ is an $(r-1)-$connected CW-complex, then $$\tc_n(X)\leq n(\mathrm{dim}(X)/r)+1.$$
\end{corollary}

Lastly, let us mention a version of K\"{u}nneth's theorem for cohomology \cite{Hat}. \\

\begin{theorem}
\label{TheoremK}
If $X$ is a CW-complex and $H^{k}(X;R)$ is a finitely generated free $R-$module for all $k$, then $H^*(X^n;R)$ and $\left( H^*(X;R)\right)^{\otimes n}$ are isomorphic as rings. \\
\end{theorem}

In our next examples, all spaces $X$ are CW-complexes and all $H^k(X;R)$ are finitely generated free $R-$modules. So the above theorem applies and we identify $\left( H^*(X;R)\right)^{\otimes n}$ with $H^*(X^n;R)$. To simplify notations, if $\alpha\in H^{\geq 1}(X;R)$, we let $\alpha\langle i\rangle$ denote $1\otimes 1\otimes \cdots \otimes 1\otimes \alpha\otimes 1\otimes \cdots\otimes 1$ in $\widetilde{H}^*(X^n;R)$, where $\alpha$ is in the $i$th position.\\

\subsection{Spheres}

Let $m\geq 1$, then the $m$-dimensional sphere $S^m$ is $(m-1)-$connected.  Thus by Corollary \ref{Corollary2}, $\tc_n(S^m)\leq n+1$. On the other hand, we have $H^*(S^m;\mathbb{Z})\cong \mathbb{Z}[\alpha]/(\alpha^2)$. So $$\prod_{i=1}^n\alpha \langle i\rangle\neq 0\in \widetilde{H}^*((S^m)^n;\mathbb{Z}).$$ Thus, $n<\nil_{\mathbb{Z}}((S^m)^n)\leq \tc_n(S^m)$ by Theorem \ref{TheoremN}. Therefore, $\tc_n(S^m)=n+1$, which is independent of the dimension of the sphere.\\

\subsection{Product of spheres}

The $k-$torus $(S^1)^k$  and the Cartesian product of $k$ 2-spheres $(S^2)^k$ model the configuration spaces of planar and spatial robotic arms with $k$ joints respectively \cite{Far1}. We first have $\cat(S^m)\leq 2$ (by Theorem \ref{TheoremW}). Then by Theorem \ref{TheoremP}, $\tc_n((S^m)^k)=\cat((S^m)^{nk})\leq nk+1$. Similar to the previous example, $nk<\nil_{\mathbb{Z}}((S^m)^{nk})\leq \tc_n((S^m)^k)$. Thus, we have $\tc_n((S^m)^k)=nk+1$, which does not depend on the dimension of the spheres, but does depend on the number of spheres.\\

\subsection{Surfaces}
Let $M^2$ be a closed and connected $2-$dimensional manifold, which is not $S^2$. So $M^2$ is either $\#_gT^2$, the connected sum of $g$ 2-tori, or $\#_gP^2$, the connected sum of $g$ projective planes, in both cases for some $g\geq 1$.\\

Since $M$ is only $0-$connected, we have  $\tc_n(M^2)\leq 2n+1$ by Corollary \ref{Corollary2}. \\

We know that $H^*(\#_gT^2;\mathbb{Z})$ is the noncommutative polynomial $\mathbb{Z} \langle \alpha_1,\beta_1,\alpha_2,\beta_2,\cdots,\alpha_g,\beta_g\rangle$ modulo the ideal  generated by the degree 2 polynomials: $\alpha_i\beta_j$ if $i\neq j$, $\alpha_k\beta_k+\beta_k\alpha_k$ for all  $k$ and $\alpha_i\alpha_j$, $\beta_i\beta_j$ for all $i,j$.
So for any $j=1,2,\cdots,g$, $$\prod_{i=1}^n\alpha_j\langle i\rangle\beta_j\langle i\rangle\neq 0\in \widetilde{H}^*((\#_gT^2)^n;\mathbb{Z}).$$ Thus by Theorem \ref{TheoremN}, $2n<\tc_n(\#_gT^2)$. \\

On the other hand, $H^*(\#_gP^2;\mathbb{Z}/2\mathbb{Z})$ is the polynomial ring $\mathbb{Z}/2\mathbb{Z}[\gamma_1,\gamma_2,\cdots,\gamma_g]$ modulo the ideal generated by the polynomials: $\gamma_i\gamma_j$ if $i\neq j$ and $\gamma_k^3$ for all $k$. So for any $j=1,2,\cdots,g$, $$\prod_{i=1}^n\gamma_j\langle i\rangle \gamma_j\langle i\rangle\neq 0\in \widetilde{H}^*((\#_gP^2)^n;\mathbb{Z}/2\mathbb{Z}).$$ By Theorem \ref{TheoremN}, $2n<\tc_n(\#_gP^2)$.\\

Therefore, if the closed and connected surface $M^2$ is not $S^2$, then $\tc_n(M^2)=2n+1$.\\

\subsection{Real projective spaces}
Since $H^*(\mathbb{R}P^m;\mathbb{Z}/2\mathbb{Z})\cong \mathbb{Z}/2\mathbb{Z}[\alpha]/(\alpha^{m+1})$, $$\prod_{i=1}^n\prod_{m}\alpha\langle i\rangle\neq 0\in \widetilde{H}^*((\mathbb{R}P^m)^n;\mathbb{Z}/2\mathbb{Z}).$$ Thus, $nm<\tc_n(\mathbb{R}P^m)$ by Theorem \ref{TheoremN}. On the other hand, notice that $\pi_1(\mathbb{R}P^m)\cong \mathbb{Z}/2\mathbb{Z}$, which is not zero. Corollary \ref{Corollary2} then implies that $\tc_N(\mathbb{R}P^m)\leq nm+1$. Therefore, $\tc_n(\mathbb{R}P^n)=nm+1$.\\

\subsection{Complex projective spaces}
We know $H^*(\mathbb{C}P^m;\mathbb{Z})\cong \mathbb{Z}[\beta]/(\beta^{m+1})$ and $\mathbb{C}P^m$ is $1-$connected. Similar to the above argument, we have $\tc_n(\mathbb{C}P^m)=nm+1$. 

\subsection{Configuration spaces of points in Euclidean spaces}

Let $m\geq 2$ and $k\geq 1$.  $F(\mathbb{R}^m,k)$, the configuration space of $k$ points in $\mathbb{R}^m$, is defined to be $$\{(x_1,x_2,\cdots,x_n)\in (\mathbb{R}^m)^k\big| x_i\neq x_j \mbox{ for }i\neq j\}.$$ It is $(m-2)-$connected and it strongly deformation retracts to a regular CW-complex $\mathcal{F}(m,k)$ of dimension $(m-1)(k-1)$ \cite{BZ}. Thus by the homotopy invariance of $\tc_n$ (Proposition \ref{PropositionInv}) and Corollary \ref{Corollary2}, $\tc_n(F(\mathbb{R}^m,k))=\tc_n(\mathcal{F}(m,k))\leq n(k-1)+1$.\\

On the other hand, the cohomology ring $H^*(F(\mathbb{R}^m,k);\mathbb{Z})$ is isomorphic to the graded commutative ring generated by degree $m-1$ elements $\alpha_{ab}$ for all $1\leq a<b\leq k$ modulo the ideal generated by $\alpha_{ab}^2$ for all $1\leq a<b\leq k$ and $\alpha_{ab}\alpha_{bc}-\alpha_{ab}\alpha_{ac}-\alpha_{ac}\alpha_{bc}$ for all $1\leq a <b<c\leq k$. See \cite{Arnold}, \cite{FCohen} and \cite{FH}. So $$\prod_{i=1}^n\prod_{a=1}^{k-1}\alpha_{a(a+1)}\langle i \rangle\neq 0\in \widetilde{H}^*((F(\mathbb{R}^m,k))^n;\mathbb{Z}),$$
from which we have $n(k-1)<\tc_n(F(\mathbb{R}^m,k))$. Therefore, $\tc_n(F(\mathbb{R}^m,k))=n(k-1)+1$.\\

Below is a table summarizing the previous results. It should be noted that the previous methods do not apply to any space. For example, the cohomology ring structure does not tell much about $\tc_n(LX)$. Computing $\tc_n(X)$ in general is more difficult and investigating them requires more technical homotopic, algebraic or analytic methods.\\

\begin{table}[!h]
\caption{$\tc_n$ for some familiar spaces.}
\begin{center}
\begin{tabular}{lllllll}
\hline
$X$           & $S^m$ & $(S^m)^k$ & $M^2\neq S^2$  & $\mathbb{R}P^m$ & $\mathbb{C}P^m$ & $F(\mathbb{R}^m,k),m\geq 2$\\
\hline
$\tc_n(X)$ & $n+1$ & $nk+1$       & $2n+1$ & $nm+1$                & $nm+1$                 & $n(k-1)+1$\\
\hline
\end{tabular}
\end{center}
\end{table}

\textbf{Acknowledgement.} The author thanks his advisor, teachers and friends at Purdue University. He also thanks Amherst College for support during the summer of 2015. 


\end{document}